\documentclass[a4paper,11pt,leqno]{amsart}

\usepackage{amsmath}
\usepackage{amsthm}
\usepackage{amsfonts}
\usepackage{amssymb}

\usepackage{tikz}
\usetikzlibrary{arrows}
\usepackage{enumerate}
\usetikzlibrary{calc}
\usepackage{colonequals}
\usepackage{verbatim}

\newcommand{\inv}{^{-1}}
\newcommand{\id}{\operatorname{id}}

\newcommand{\cof}{\operatorname{cof}}

\newcommand{\R}{\mathbb R}
\newcommand{\bS}{\mathbb S}

\newcommand{\abs}[1]{\lvert#1\rvert}
\newcommand{\norm}[1]{\lVert#1\rVert}
\DeclareMathOperator{\esssup}{ess \, sup}

\newtheorem{theorem}{Theorem}

\newtheorem*{theorem*}{Theorem}{\bf}{\it}

\newtheorem*{proposition*}{Proposition}{\bf}{\it}

\newtheorem*{lemma*}{Lemma}{\bf}{\it}
\newtheorem{corollary}[equation]{Corollary}
\theoremstyle{definition}

\newtheorem*{definition*}{Definition}
\theoremstyle{remark}

\newtheorem{remarks}[equation]{Remarks}

\begin{document}

\title{On BLD-mappings with small distortion}

\author{Aapo Kauranen}
\address{
  Department of Mathematics and Statistics, P.O. Box 35, FI-40014 University of Jyv\"askyl\"a, Finland
}
\email{aapo.p.kauranen@jyu.fi}
\thanks{A.K.
  acknowledges the support of Academy of Finland, grant number 322441.}

\author{Rami Luisto}
\address{Department of Mathematics and Statistics, P.O. Box 35, FI-40014 University of Jyv\"askyl\"a, Finland}
\email{rami.luisto@gmail.com}
\thanks{R.L.\
  was partially supported by the Academy of Finland
  (grant 288501 `\emph{Geometry of subRiemannian groups}')
  and by the European Research Council
  (ERC Starting Grant 713998 GeoMeG `\emph{Geometry of Metric Groups}').
}

\dedicatory{Dedicated to Professor Pekka Koskela on his $59$th birthday}

\author{Ville Tengvall}
\address{Department of Mathematics and Statistics, P.O. Box 68 (Pietari Kalmin katu 5), FI-00014 University of Helsinki, Finland}
\email{ville.tengvall@helsinki.fi}
\thanks{The research of V.T. was supported by the Academy of Finland, project number 308759.}

\date{\today}

\begin{abstract}
  We show that every $L$-BLD-mapping in a domain of $\R^n$ is a local homeomorphism if $L < \sqrt{2}$ or $K_I(f) < 2$. These bounds are sharp as shown by a winding map.
\end{abstract}

\maketitle


\section{Introduction}

Mappings with $L$-bounded length distortion (abbr. \emph{$L$-BLD mappings}) were originally introduced by Martio and V\"{a}is\"{a}l\"{a} in \cite{MartioVaisala} as continuous, sense-preserving, discrete, and open mappings
$$f \colon \Omega  \to \R^n \quad \text{($\Omega \subset \R^n$ domain with $n \ge 2$)}$$
satisfying the upper and lower length distortion bounds
\begin{align}\label{LengthDistortion}
  L\inv\ell(\alpha)
  \leq \ell(f \circ \alpha)
  \leq L \ell(\alpha)
\end{align}
for every path $\alpha$ in $\Omega$ and for a fixed constant $L \ge 1$, where $\ell(\gamma)$ denotes the length of a path $\gamma$. Notice that no constant map satisfies \eqref{LengthDistortion}.  These mappings form a  superclass of local $L$-bi-Lipschitz mappings which have been studied by several authors, see e.g.\ \cite{DrasinPankka, HajMalZim18, LeDPank14, HeiSul02} and the references therein. If a mapping is an $L$-BLD-mapping for some $L \ge 1$ we may simply call it a \emph{BLD-mapping}. 

Unlike local bi-Lipschitz maps, BLD-mappings do not need to be local homeomorphisms. Indeed, for instance the map
\begin{align}\label{MainExample}
  (r,\theta, z) \stackrel{w}{\mapsto} \bigl( \sqrt{2}^{-1}r,2\theta, \sqrt{2}^{-1}z  \bigr) \quad \quad \text{($z \in \R^{n-2}$)}
\end{align}
in cylindrical coordinates of $\R^{n}$ defines a noninjective $\sqrt{2}$-BLD mapping. Our first theorem shows that this mapping is extremal for the failure of local homeomorphism property:

\begin{theorem}\label{thm:Main2}
  Every $L$-BLD-mapping  
  $$f \colon \Omega \to \R^n \quad (\text{$\Omega \subset \R^n$ domain with $n \ge 2$})$$
  with $L < \sqrt{2}$ is a local homeomorphism.
\end{theorem}

Mappings of $L$-bounded length distortion form also a subclass of $L^{2(n-1)}$-quasiregular mappings, see \cite[Lemma~2.3]{MartioVaisala}. We recall that a mapping
$$f \colon \Omega  \to \R^n \quad \text{($\Omega \subset \R^n$ domain with $n \ge 2$)}$$
is called \emph{$K$-quasiregular} if it belongs to Sobolev space $W_{loc}^{1,n}(\Omega, \R^n)$ and satisfies the following \emph{distortion inequality}
\begin{align*}
  \abs{Df(x)}^n \le K J_f(x) \quad \text{a.e.}
\end{align*}
for a given constant $K \ge 1$. Here and what follows $\abs{A}$ stands for the \textit{operator norm} of a given $n \times n$ matrix $A$ and 
$$J_f(x) = \det Df(x)$$ 
denotes the \textit{Jacobian determinant} of the differential matrix $Df(x)$. To every quasiregular mapping we associate the well-known \emph{inner} and \emph{outer distortion functions} defined as follows 
\begin{align*}
  K_I(x,f)
  =  \left\{ \begin{array}{ll}
               \frac{\abs{D^\#f(x)}^{n}}{J_f(x)^{n-1}}, & \textrm{if $J_f(x)>0$}\\
               1, & \textrm{otherwise}
             \end{array} \right.
\end{align*}
and
\begin{align*}
  K_O(x,f)
  =  \left\{ \begin{array}{ll}
               \frac{\abs{Df(x)}^n}{J_f(x)}, & \textrm{if $J_f(x)>0$}\\
               1, & \textrm{otherwise,}
             \end{array} \right.
\end{align*}
measuring how the infinitesimal geometry of $n$-dimensional balls is distorted by the mapping. Here and what follows 
\begin{align*}
  D^{\sharp}f(x) = (\cof Df(x))^T,
\end{align*}
where $\cof$ stands for cofactor matrix. The corresponding \emph{inner} and \emph{outer dilatations} are defined by 
\begin{align*}
  K_I(f) \colonequals \underset{x \in \Omega}{\esssup} \, K_I(x,f),
  \quad
  \text{and}
  \quad
  K_O(f) \colonequals \underset{x \in \Omega}{\esssup} \, K_O(x,f) \,.
\end{align*}
For the theory and basic properties of quasiregular mappings we refer to the standard monographies \cite{IwaniecMartin, Rickman-book, Reshetnyak67, Vuorinen}.

By the generalized Liouville theorem of Gehring \cite{Gehring1962} and Reshetnyak \cite{Reshetnyak1967} non-constant 1-quasiregular mappings are restrictions of M\"{o}bius transformations. On the other hand, for the winding map in \eqref{MainExample} we have
$$K_I(w) = 2 \quad \text{and} \quad K_O(w) = 2^{n-1}.$$
In \cite[Remark~4.7]{MRV-71} it was conjectured that any quasiregular mapping with $K_I(f)<2$ is a local homeomorphism. That is, the behaviour of the winding map is extremal in this sense.

The conjecture is not yet solved, but some partial results are known.
By the result of Martio, Rickman, and V\"{a}is\"{a}l\"{a} \cite[Theorem~4.6]{MRV-71} there exists $\varepsilon(n) > 0$ such that every quasiregular mapping
$$f \colon \Omega \to \R^n \quad \text{($\Omega \subset \R^n$ with $n \ge 3$)} \quad \text{with $K_I(f) < 1+\varepsilon(n)$}$$
is a local homeomorphism, see also \cite{Goldstein1971}. The best known quantitative bound for the number $\varepsilon(n)>0$ is given by Rajala in \cite{Rajala-MartioResult}. Furthermore, the conjecture is known to be true when the \emph{branch set} 
$$B_f = \{ x \in \Omega : \text{$f$ is not a local homeomorphism at $x$} \}$$
of the mapping is geometrically nice. In particular, this is the case when the branch set contains a rectifiable curve, see e.g. \cite[p. 76]{Rickman-book}. 
We give a short proof showing that the conjecture holds also for $L$-BLD mappings. See also \cite{HeKi} for other related injectivity results for BLD mappings.
\begin{theorem}\label{thm:Main1}
  Let 
  $$f \colon \Omega \to \R^n \quad (\text{$\Omega \subset \R^n$ domain with $n \ge 2$})$$
  be an $L$-BLD-mapping. 
  Then 
  \begin{align}\label{MartioIneq}
    K_I(f) \geq i(x,f) \quad \text{for every $x \in \Omega$.}
  \end{align} Especially, if $K_I(f)<2$ then $f$ is a local homeomorphism.
\end{theorem}
Notice that the conclusion of Theorem \ref{thm:Main2} does not hold for quasiregular mappings in general. In \cite{MRV-71} there is an example of a quasiregular mapping $f\colon\Omega\to\R^n$ for which $\sup\{i(x,f)\colon x\in\Omega\}=\infty.$ Also, notice that Theorem \ref{thm:Main2} is valid in dimension two but the same conclusion does not hold for planar quasiregular mappings as one can see by considering the holomorphic function 
$$f \colon \mathbb{C} \to \mathbb{C}, \quad f(z) = z^2.$$ This function is obviously not a BLD-mapping. We also point out that the proof of Theorem~\ref{thm:Main1} is based on the modulus of continuity of quasiregular mappings and on the $L$-radiality of BLD-mappings. The BLD-property is used here only for the $L$-radiality. 


Finally, we recall that by Zorich's theorem \cite{Zorich} every locally injective, entire quasiregular mapping in dimension $n \ge 3$ is quasiconformal. Zorich's theorem fails in the plane, which is shown by the mapping $z \mapsto \exp(z)$. Note that in the BLD-setting every entire, locally homeomorphic BLD-mapping in $\R^n$, $n \ge 2$, is a bi-Lipschitz homeomorphism onto $\R^n$. See eg. \cite[Lemma~4.3]{MartioVaisala}.  If we combine this with our main results we obtain the following:

\begin{corollary}\label{Cor:GlobalHomeo}
  Let $f \colon \R^n \to \R^n$, $n \ge 2$, be an entire BLD-mapping with
  $$L < \sqrt{2} \quad \text{or} \quad K_I(f) < 2.$$
  Then $f$ is an $L$-bi-Lipschitz map onto $\R^n$.
\end{corollary}


\section{Proof of Theorem~\ref{thm:Main2}}

In the proof of Theorem~\ref{thm:Main2} we use the $L$-radiality property of mappings with $L$-bounded length distortion, see \cite{Luisto-Characterization}. For a space $X$ we use the following notation:
\begin{align*}
  H_j(X) &= \text{``$j$th homology group of $X$"}\\
  H^j(X) &= \text{``$j$th Alexander-Spanier co-homology group of $X$''.}
\end{align*}
For definition and basic properties of these homology and co-homology groups, as well as the suspension mappings
used in the following proof, we refer to \cite{Hatcher}.
Necessary information on the topological degree theory used in the proof can be found from \cite[Chapter~I]{Rickman-book}.

\begin{proof}[Proof of Theorem~\ref{thm:Main2}]
  Let $f \colon \Omega \to \R^n$ be an $L$-BLD-mapping with $L < \sqrt{2}$. For the contradiction, suppose that the branch set 
  \begin{align*}
    B_f = \{ x \in \Omega : \text{$f$ is not a local homeomorphism at $x$} \}
  \end{align*}
  of $f$ is nonempty and fix a point $x_0 \in B_f$. Without loss of generality
  we may assume that $x_0 = \mathbf{0} = f(x_0)$.
  
  We take first the blow-up of $f$ by defining a sequence of mappings 
  $$g_j \colon j\Omega \to \R^n, \quad g_j(x) = j f(x/j) \, ,$$
  where 
  $$j\Omega \colonequals  \{ jx \in \R^n \mid x \in \Omega\} \quad \quad \text{($j=1,2, \ldots$)} \, .$$ We note that by \ \cite[Theorem 4.7]{MartioVaisala} (see also \cite[Section 4]{Luisto-Characterization}) the
  mappings $g_j$ are $L$-BLD-mappings and contain a subsequence converging uniformly
  to an $L$-BLD-mapping $g \colon \R^n \to \R^n$ such that $$\mathbf{0} \in B_g, \quad g(\mathbf{0}) = \mathbf{0}, \quad \text{and} \quad g \inv (\{ g(\mathbf{0}) \}) = \mathbf{0}.$$
  Furthermore, since the $L$-BLD-mapping $f$ is $L$-radial at $\mathbf{0},$ that is, there exists $r_0>0$ such that for every $x\in B(0,r_0)$ we have
  $$
  L \inv \| x \|
    \leq \| f(x) \|
    \leq L \| x \|.
  $$
  Clearly, for $g_j$ we have
  $$
  L \inv \| x \|
    \leq \| g_j(x) \|
    \leq L \| x \|
  $$
  for all $x\in B(0,jr_0).$
    
  This implies that the blow-up mapping $g$ satisfies
  \begin{align*}
    L \inv \| x \|
    \leq \| g(x) \|
    \leq L \| x \|
  \end{align*}
  for all $x \in \R^n$. In particular, we note that the composition of $g|_{\partial B(\mathbf{0},r)}$
  and the radial projection map
  \begin{align*}
    p \colon \R^n \setminus B(\mathbf{0},r/L) \to \partial B(\mathbf{0},r/L),
    \quad
    p(x) = (r/L) \frac{x}{\|x\|}
  \end{align*}
  is $L$-Lipschitz. Thus, the mapping
  \begin{align*}
    h \colon \bS^{n-1} \to \bS^{n-1},
    \quad
    h(x) = (L/r) (p \circ g|_{\partial B(\mathbf{0},r)})(rx)
  \end{align*}
  is $L^2$-Lipschitz.
  Now by a classical dilatation result \cite[Proposition 2.9, p.\ 30]{Gromov} under the assumption $L < \sqrt{2}$ we see that the induced homomorphism 
  $$h_\ast \colon H_{n-1}(\bS^{n-1}) \to H_{n-1}(\bS^{n-1})$$
  equals either $\pm \id$ or the constant map. 
  
  Since for all $r>0$ the image of the restriction $g|_{\partial B(\mathbf{0},r)}$
  avoids the origin, we note that the restrictions 
  are in fact mutually homotopic in $\R^n \setminus \{ 0 \}$ and homotopic to $h$. In particular
  since $\bS^{n-1}$ is a homotopy retract of $\R^n \setminus \{ 0 \}$,
  we see that $g|_{\R^n \setminus \{0\}}$ is homotopic to $h \times \id_{\R}$ with the identification
  $\R^n \setminus \{ 0 \} \simeq \bS^{n-1} \times \R$.
  
  Finally we note that as a BLD-mapping from $\R^n$ to $ \R^n$, the mapping $g$ extends into
  a branched cover $\hat g \colon \bS^n \to \bS^n$, see e.g.\ \cite{Luisto-NoteOnLocalToGlobal}.
  Furthermore, by the above arguments this mapping $\hat g$ is homotopic to the suspension of
  the map $h$. Therefore, the induced homomorphism 
  $$\hat g_\ast \colon H_n(\bS^n) \to H_n(\bS^n)$$ 
  equals either $\pm \id$ or the constant map.
  But now by the universal coefficient theorem \cite[p.190]{Hatcher} the induced homomorphism
  $$\hat g^\ast \colon H^n(\bS^n) \to H^n(\bS^n)$$ 
  also equals either $\pm \id$ or a constant map.
  As the degree of a branched cover $\bS^n \to \bS^n$ is always non-zero, this implies
  that $\hat g$ has degree $\pm 1$. Thus, the mapping $\hat g : \bS^n \to \bS^n$ is injective. This is a contradiction 
  with the assumption $\mathbf{0} \in B_f$ and so the original claim holds.
\end{proof}

\section{Proof of Theorem~\ref{thm:Main1}}

In what follows, for a given continuous mapping 
$$f \colon \Omega \to \R^n \quad \text{($\Omega \subset \R^n$ domain with $n \ge 2$)}$$
we denote its \emph{local (topological) index} at a point $x \in \Omega$ by $i(x,f) \in \mathbb{Z}$. We recall that as a sense-preserving, continuous, discrete and open mapping every BLD-mapping satisfies
$$i(x,f) \ge 1 \quad \text{for every $x \in \Omega$.}$$
Moreover, for a given point $x \in \Omega$ we have $i(x,f) = 1$ if and only if $f$ is a local homeomorphism. Thus, in order to proof Theorem~\ref{thm:Main1} it suffices to show that under the assumptions of the theorem we have $K_I(f) \ge i(x,f)$ for every $x \in \Omega$. For the properties of local index used here and what follows, see \cite[Chapter~1]{Rickman-book}.

\begin{proof}[Proof of Theorem~\ref{thm:Main1}]
  Fix a point $x_0 \in \Omega$ and denote
  $$\mu \colonequals  \bigl(i(x_0,f)/K_I(f) \bigr)^{\frac{1}{n-1}}.$$
  By combining \cite[Corollary 2.13]{MartioVaisala} and \cite[Theorem III.4.7, p.72]{Rickman-book} we see that there exists a radius $r>0$ and a constant $B \ge 1$ such that
  \begin{align}\label{MartioCondradiction}
    L^{-1} \norm{x_0-y} \le \norm{f(x_0)-f(y)} \le B \norm{x_0-y}^{\mu}
  \end{align}
  for every $y \in B(x_0,r)$. By letting $y \to x_0$ we get from \eqref{MartioCondradiction} that $\mu \le 1$. Especially, from here it follows that
  \begin{align}\label{WeUseThisNext}
    K_I(f) \ge i(x,f) \quad \text{for every $x \in \Omega$,}
  \end{align}
  and the claim follows. 
\end{proof}

\begin{remarks} We end this note with the following remarks:
  \begin{itemize}
  \item[(a)] In the planar case Theorem~\ref{thm:Main1} implies  Theorem~\ref{thm:Main2}. Indeed, if 
    $$f \colon \Omega  \to \R^2 \quad \text{($\Omega \subset \R^2$ domain)}$$ 
    is an $L$-BLD-mapping with $L < \sqrt{2}$ then by \cite[Lemma~2.3]{MartioVaisala} it is $K$-quasiregular with 
    $$K \le L^{2(n-1)}  < 2 \, .$$
    Especially, in the planar case we have
    $$K_I(f) = K_O(f) < 2.$$
    Thus, by Theorem~\ref{thm:Main1} the mapping $f$ is a local homeomorphism.
    
  \item[(b)] We do not know what is the optimal outer dilatation bound for the local injectivity of BLD-mappings. However, the map
    $$(r, \theta, z) \mapsto (r,2\theta, 2z) \quad \quad (z \in \R^{n-2})$$
    in cylindrical coordinates would suggest that every BLD-mapping with $K_O(f) < 2$ is a local homeomorphism.
    
  \item[(c)] In \cite{KLT} we obtained that for every $n \ge 2$ there exists a quasiregular mapping of the unit ball with a compact and nonempty branch set. Moreover, this mapping can be actually assumed to be a BLD-map. It would be interesting to find the optimal length distortion and outer and inner dilatation bounds for the existence of nonempty compact branch sets for BLD-mappings of the unit ball.
  \end{itemize} 

\end{remarks}




\def\cprime{$'$}

\end{document}